\documentclass[letterpaper, reqno, 11pt]{amsart}

\usepackage{amsfonts}
\usepackage{amsmath, amsthm, amssymb}
\usepackage{enumerate, url}
\usepackage{hyperref}
\usepackage{mathrsfs}
\usepackage{thmtools}
\usepackage{thm-restate}
\usepackage{caption}
\usepackage{mleftright}
\usepackage{mathtools}

\newtheorem{theorem}{Theorem}[section]
\newtheorem{lemma}[theorem]{Lemma}

\newtheorem{corollary}[theorem]{Corollary}
\newtheorem{proposition}[theorem]{Proposition}
\newtheorem{Prop}[theorem]{Proposition}
\newtheorem{prop}[theorem]{Proposition}

\theoremstyle{definition}
\newtheorem{definition}[theorem]{Definition}

\newtheorem{example}[theorem]{Example}

\theoremstyle{definition}
\newtheorem{question}[theorem]{Question}

\theoremstyle{remark}

\newtheorem{Rmk}[theorem]{Remark}

\newcommand{\Ga}{\Gamma}
\newcommand{\ga}{\gamma}

\newcommand{\La}{\Lambda}
\newcommand{\cal}{\mathcal}
\newcommand{\al}{\alpha}
\newcommand{\be}{\beta}

\title{Quasi-Isometric embeddings of Ramanujan complexes}
\author{Hyein Choi}
\address{Department of Mathematics, Rice University, Houston, TX, USA}
\email{hc71@rice.edu}

\begin{document}

\begin{abstract}
Ramanujan complexes were defined as high dimensional analogs of the optimal expanders, Ramanujan graphs. They were constructed as quotients of the Euclidean building (also called the affine building and the Bruhat-Tits building) of $\mathrm{PGL}_d(\mathbb{F}_p((y)))$ by certain cocompact lattices 
by Lubotzky-Samuels-Vishne. We distinguish the Ramanujan complexes up to large-scale geometry. More precisely, we show that if $p$ and $q$ are distinct primes, then the associated Ramanujan complexes do not quasi-isometrically embed into one another. The main tools are the box space rigidity of Khukhro-Valette and the Euclidean building rigidity of Kleiner-Leeb and Fisher-Whyte.
\end{abstract}

\maketitle

\tableofcontents

\section{Introduction}

\subsection{Main result}

Let $d \geq 3$, $p$ be a prime, and $G=\mathrm{PGL}_d(\mathbb{F}_p((y)))$, where $\mathbb{F}_p((y))$ is the local field of Laurent series over the finite field $\mathbb{F}_p$. 
Let $\cal{B}_p$ be the Euclidean building of $G$ which is also called the \textit{affine building} and the \textit{Bruhat-Tits building} considered as an analogue of symmetric spaces.
Denote by $\mathscr{X}_p= \{ \Gamma(I_i)\backslash \mathcal{B}_p \}_i$ an infinite family of Ramanujan complexes constructed as in \cite{lubotzky2005explicit}, 
where $\Ga$ is a cocompact lattice of $G$ and $\Ga(I_i)$ is a finite index normal subgroup of $\Ga$ for each $i$ satisfying certain conditions.
The background on Ramanujan complexes is discussed in
Section \ref{intronra} and we will briefly review the construction of such $\Ga$ and $\Ga(I_i)$ in Section \ref{ra}.

Our main result is to distinguish families of Ramanujan complexes of equal rank associated with distinct primes up to quasi-isometric embeddings.

\begin{theorem} [Main result] \label{3}
    Let $d \geq 4$ and $\mathscr{X}_p= \{ \Gamma(I_i)\backslash \mathcal{B}_p \}_i$ be an infinite family of Ramanujan complexes built from the Euclidean building $\cal{B}_p$ of $\mathrm{PGL}_d(\mathbb{F}_p((y)))$ as above. For distinct primes $p$ and $q$, $\mathscr{X}_p$ does not quasi-isometrically embed in $\mathscr{X}_q$.
 \end{theorem}

In particular, our work is parallel to the work of Khukhro-Valette \cite{khukhro2017expanders}. They consider \textit{box spaces}
$\{\mathrm{Cay}(G/N_i)\}_i$ of $G$ that are families of Cayley graphs of $G/N_i$, where $G$ is a finitely generated residually finite group and $\{N_i\}$ is a decreasing sequence of finite index normal subgroups of $G$ with trivial intersection.

\begin{theorem} \label{kvmain} \cite[Theorem C-($1$)]{khukhro2017expanders}
    Let $X_p$ be a box space of $\mathrm{SL}_2(\mathbb{Z}[\sqrt{p}])$.
    For distinct primes $p$ and $q$, $X_p$ is not coarse equivalent\footnote{Coarse equivalence is defined similarly to quasi-isometry whose control functions in (\ref{qi}) need not be affine but diverge} to $X_q$.
\end{theorem}

Theorem \ref{kvmain} can be viewed as a $1$-dimensional analogue of our result since the objects are built from $\rm{SL}_2$ and $\rm{PGL}_d$ for $d \geq 4$, respectively.
Indeed, Theorem \ref{kvmain} will produce expander graphs and Theorem \ref{3} will produce high dimensional expanders in the following sections Corollary \ref{kvmain2} and Corollary \ref{coro}. The main difference is that we distinguish Ramanujan complexes up to quasi-isometric embeddings rather than coarse equivalence. 
While neither notion is stronger than the other, they differ in nature: coarse equivalence is an equivalence relation, whereas quasi-isometric embedding is not.
Thus our distinction requires ruling out quasi-isometric embeddings in both directions rather than excluding a single equivalence.

Corollary \ref{sym vs euc} is a special case of a key ingredient to prove Theorem \ref{3}. It is not used in proving the main theorem but it may be of independent interest as an answer to a natural question about how symmetric spaces differ from their $p$-adic analogue Euclidean buildings.

\begin{corollary} [to Proposition \ref{1}] \label{sym vs euc}
    Let $n \geq 4$ and $p$ be a prime.
    Let $X$ be the symmetric space of $\mathrm{SL}_n(\mathbb{R})$ and $Y_p$ be the Euclidean building of $\mathrm{SL}_n(\mathbb{Q}_p)$.
    \begin{enumerate}
        \item  There is no quasi-isometric embedding $X \to Y_p$ nor $Y_p \to X$.
        \item There is no quasi-isometric embedding
        $Y_p \to Y_q$ for distinct primes $p$ and $q$.
    \end{enumerate}
\end{corollary}

\subsection{Motivation from expanders}
The above theorems are inspired by \textit{families of expander graphs} which have many applications in mathematics and computer science due to their distinctive property, namely highly connected (sharing a common Cheeger constant) but also sparse (finite graphs of uniformly bounded degree whose numbers of vertices go to infinity) \cite{lubotzky1994discrete, hoory2006expander}. An intuitive real-life application of expanders is the economical placement of telephone towers for effective coverage.

Our main result is motivated by the following question about expanders which is posed by Mendel-Naor \cite[Section 1.1 and Section 9]{mendel2014nonlinear}, Ostovskii \cite [Theorem 5.76]{ostrovskii2013metric}, and at an Oberwolfach workshop \cite{khukhro2019mini}. The question will be discussed in more detail in Section \ref{5}.
\begin{question}\label{question} 
Distinguish expanders up to large-scale geometry.
\end{question}
There have been a lot of results in this direction constructing and distinguishing expanders and even super-expanders up to large-scale geometry \cite{khukhro2017expanders, das2018geometry, de2019superexpanders,fisher2019rigidity,nowak2017warped,vigolo2019discrete}, etc. 

Since Margulis first gave an explicit construction of expanders \cite{gregory1973margulis}, it is known that box spaces of groups with Kazhdan’s Property (T), and more generally with Property ($\tau$), form families of expander graphs. Hence Theorem \ref{kvmain} can be reformulated from the perspective of Question \ref{question} as follows.
\begin{corollary} [to Theorem \ref{kvmain}] \label{kvmain2}
    Let $X_p$ be a box space of $\mathrm{SL}_2(\mathbb{Z}[\sqrt{p}])$.
    As $p$ varies, $X_p$'s form families of expander graphs which are not coarse equivalent to each other. 
\end{corollary}


\subsection{Ramanujan: the optimal expanders}\label{intronra}
A \textit{$k$-regular Ramanujan graph} is a 
$k$-regular graph whose eigenvalues $\lambda$ of 
the adjacency matrix satisfy either 
$|\lambda|=k$ (trivial eigenvalues) or 
$|\lambda| \leq 2 \sqrt{k-1}$ (non-trivial eigenvalues). 
Lubotzky-Philips-Sarnak constructed $k$-regular Ramanujan graphs as quotients of the Euclidean building of $\rm{PGL}_2(F)$, which is a tree, by cocompact lattices of $\rm{PGL}_2(F)$ \cite{lubotzky1988ramanujan}.
In other words, the non-trivial spectrum of Ramanujan graphs is bounded by the spectrum of their covering tree. 
A remarkable feature is that families of Ramanujan graphs form families of expander graphs with the optimal spectrum in the sense of Alon-Boppana \cite{alon1986eigenvalues} on which the random walk is distributed in the fastest way (see \cite[Section 5]{lubotzky2020ramanujan} and the references therein). 

Generalizing this to high dimensions, Lubotzky-Samuels-Vishne defined and constructed \textit{Ramanujan complexes} as quotients of the Euclidean building of $\rm{PGL}_d(F)$ for $d \geq 3$ by cocompact lattices of $\rm{PGL}_d(F)$ whose non-trivial spectrum is derived from their covering Euclidean buildings \cite{lubotzky2005explicit}. 
Similar to 1-dimension, families of Ramanujan complexes have the optimal spectrum in the sense of high dimensional Alon-Boppana shown by Li \cite{li2004ramanujan}.

One might expect that Ramanujan complexes form families of ``high dimensional expanders'' ideally exhibiting the distinctive property of expander graphs not only through edges (1-dimension) but also through triangles (2-dimension), tetrahedra (3-dimension), and so on. However, high dimensional expanders are not yet defined and we only have several proposed concepts. See \cite{lubotzky2014ramanujan}, \cite{lubotzky2018high} for candidates of a definition and how they are related to each other.

Among them, we focus on \textit{Gromov high dimensional expanders} since they meet the natural expectation as follows. By a \textit{family of Gromov geometric (resp.\@ topological) high dimensional expanders},
we mean a collection $\{X_i\}$ of
finite $n$-dimensional simpicial complexes of 
uniformly bounded degree satisfying the \textit{Gromov $c$-overlapping property} for some fixed $c>0$ defined in \cite{gromov2010singularities}, that is, 
for any affine (resp.\@ continuous) map 
$f:X_i \to \mathbb{R}^n$,
there exists $z \in \mathbb{R}^n$ which is covered by $c$-proportion of the images of $n$-simplicies under $f$.

\begin{theorem} [geometric high dimensional expanders \cite{fox2012overlap, evra2017finite}] \label{geo}
    Let $\mathscr{X}_p= \{ \Gamma(I_i)\backslash \mathcal{B}_p \}_i$ be a family of Ramanujan complexes built from the Euclidean building $\mathcal{B}_p$ of $\mathrm{PGL}_d(\mathbb{F}_p((y)))$.
    For each $d \geq 3$, there exists $p_0=p_0(d) > 0$ such that $\mathscr{X}_p$ forms a family of Gromov geometric high dimensional expanders for every prime $p > p_0$. 
\end{theorem}

\begin{theorem}[topological high dimensional expanders \cite{kaufman2016isoperimetric, evra2024bounded}] \label{top}
    Let $\mathscr{X}_p= \{ \Gamma(I_i)\backslash \mathcal{B}_p \}_i$ be a family of Ramanujan complexes built from the Euclidean building $\mathcal{B}_p$ of $\mathrm{PGL}_d(\mathbb{F}_p((y)))$. For each $d \geq 4$, there exists $p_0=p_0(d) > 0$ such that the family of $(d-2)$-skeleta of the $(d-1)$-dimensional simiplicial complexes of $\mathscr{X}_p$ forms a family of Gromov topological high dimensional expanders for every prime $p > p_0$.
\end{theorem}

Combining our main theorem with these results, we can observe a high dimensional analogue of Corollary \ref{kvmain2} as distinct primes produce significantly different types of optimal high dimensional expanders.
\begin{corollary} [to Theorem \ref{3}] \label{coro} 
    For each $d \geq 4$, there exists $p_0 =p_0(d)>0$ such that $\mathscr{X}_p$'s (resp.\@ the codimension one skeleta of $\mathscr{X}_p$'s) form infinite families of Gromov geometric (resp.\@ topological) high dimensional expanders 
    which do not quasi-isometrically embed into one another 
    for every prime $p > p_0$.
\end{corollary}

We refer to \cite{lubotzky2014ramanujan, lubotzky2018high, lubotzky2020ramanujan} for Ramanujan graphs, Ramanujan complexes, and high dimensional expanders.

\subsection{Proof strategy}
The strategy to prove Theorem \ref{3} is similar to  \cite{khukhro2017expanders}. The first step is to show a large-scale geometry rigidity, and the second step is to produce coarsely disjoint expanders from different primes $p$.
The first step is done by Proposition \ref{2} basically
saying that quasi-isometric embeddings of finite quotients of Euclidean buildings induce quasi-isometric embeddings of the Euclidean buildings. Since box spaces and finite quotients of Euclidean buildings have similar structure, the first step is obtained naturally by translating Khukhro-Valette's Proposition \ref{KV} for box spaces in terms of Euclidean buildings.

For the second step, while Khukhro-Valette apply the quasi-isometry rigidity of Farb-Schwartz to the lattices $\mathrm{SL}_2(\mathbb{Z}[\sqrt{p}]) $\cite{farb1996large}, we use Proposition \ref{1} 
showing that two Euclidean buildings (or even symmetric spaces) from ``different'' fields do not quasi-isometrically embed into one another. Proposition \ref{1} follows from the quasi-isometry/quasi-isometric embedding rigidity of
Kleiner-Leeb and Fisher-Whyte \cite{kleiner1997rigidity}, \cite{fisher2018quasi} together with the classical building theory. While we do not use it, the work of Eskin-Farb is also relevant \cite{eskin1997quasi}.

\medskip 
\textbf{Acknowledgement.} The author is deeply grateful to David Fisher for valuable insights, stimulating discussions, and constant support throughout the process. The author thanks Jean-François Lafont and Wouter Van Limbeek for valuable comments and helpful discussions on Section 5. 
The author also thanks Thang Nguyen and Ralf Spatzier for careful reading and valuable comments. Further thanks go to Homin Lee, Insung Park, and Brian Udall for providing helpful feedback on the earlier version. This project was supported in part by NSF under grant DMS-2246556 and by Institut Henri Poincaré (UAR 839 CNRS-Sorbonne Université).

\section{Preliminaries}

In this section, we review some terminology and properties that we will use. We first recall the definitions of quasi-isometries and quasi-isometric embeddings for spaces and infinite families of spaces.
\begin{definition} \label{def}
    Let $X$ and $Y$ be metric spaces. For $L\geq1$ and $C\geq 0$, a map $\phi: X \to Y$ is a \textit{$(L,C)$-quasi-isometric embedding} if 
    
    \begin{equation}\label{qi}
       \frac{1}{L}d(a,b)-C \leq d(\phi(a),\phi(b)) \leq Ld(a,b)+C 
    \end{equation}
    for every $a,b \in X$. 
    The map $\phi$ is a \textit{$(L,C)$-quasi-isometry} if it is a $(L,C)$-quasi-isometric embedding and \textit{coarsely surjective}, that is, for every $y \in Y$, there exists $x \in X$ such that $d(y,\phi(x))\leq C$. 

    Let $\mathscr{X}=\{X_i\}_i$ and $\mathscr{Y}=\{Y_i\}_i$ be infinite families of metric spaces. We say that $\mathscr{X}$ \textit{quasi-isometrically embeds} in (resp.\@ is \textit{quasi-isometric} to) $\mathscr{Y}$ if there exist $L \geq 1$, $C \geq 0$, and $f_i: X_i \to Y_i$ such that $f_i$ is a $(L,C)$-quasi-isometric embedding (resp.\@ $(L,C)$-quasi-isometry) for every $i$.
\end{definition}

\subsection{Definitions of buildings}
For a concrete background on spherical buildings, Euclidean buildings, and their relationship,
we refer to \cite{abramenko2008buildings} for a classical viewpoint and \cite{kleiner1997rigidity}, \cite{MR1744486} for a geometric viewpoint. 

\medskip
\textbf{From a classical point of view} 
A \textit{ spherical (resp.\@ Euclidean) building} is a simplicial complex with subcomplexes called \textit{apartments} which are spherical (resp.\@ Euclidean) Coxeter complexes associated with a finite (resp. infinite affine) Coxeter group $W$ (resp.\@ $W_{\rm{aff}}$) satisfying the following axioms:

\begin{enumerate}
        \item (Plenty of apartments) Any two simplices in the building are contained in an apartment.
        \item (Compatible apartments) For any two apartments $\Sigma_1$ and $\Sigma_2$, and any top-dimensional simplices called \textit{chambers}
        $C_1$ and $C_2$ contained in $\Sigma_1 \cap \Sigma_2$, there exists a simplicial isometry from $\Sigma_1$ to $\Sigma_2$ which fixes $C_1$ and $C_2$ pointwise. 
     \end{enumerate}

Each apartment can be thought of as being modeled on the sphere $\mathbb{S}^{n-1}$ (resp.\@ Euclidean space $\mathbb{R}^n$) triangulated by the reflection hyperplanes of the Coxeter group (resp.\@ affine Coxeter group).

\begin{definition}
    Let $X$ be a building. A subcomplex $X'$ of $X$ is a \textit{subbuilding} of $X$ if it is a building itself and every apartment of $X'$ is an apartment in $X$.
\end{definition}

\begin{prop} \cite[Proposition 4.63]{abramenko2008buildings}
    Every subbuilding $X'$ of a building $X$ has the same Coxeter matrix as $X$.
\end{prop}

In other words, subbuildings share the same Coxeter structure as the ambient buildings with possibly losing some information.

\medskip
\textbf{From a geometric point of view} 
Buildings are equipped with the metrics from their apartments.
A spherical (resp.\@ Euclidean) building is a $\rm{CAT}(1)$ (resp.\@ $\rm{CAT}(0)$)-space denoted by $\Delta$ (resp.\@ $\cal{B}$) together with the \textit{charts} $i:\mathbb{S}^{n-1} \to \Delta$ 
(resp.\@ $i:\mathbb{R}^{n} \to \cal{B}$), which are isometric embeddings and each $i(\mathbb{S}^{n-1})$ (resp.\@ $i(\mathbb{R}^n)$) is an \textit{apartment} in $\Delta$ (resp.\@ $\cal{B}$).

In particular, each Euclidean building has an additional structure, the \textit{directions of geodesic segments} ``looking at infinity'' satisfying two more axioms (we refer to \cite[Section 4.1.2]{kleiner1997rigidity}):
Consider an apartment $E$ in $\cal{B}$ and its \textit{Tits boundary} $\partial_T E$, which is the set of all equivalence classes of geodesic rays equipped with the Tits metric. We have a map
$\rho: \rm{Isom}(E) \to \rm{Isom}(\partial_T E)$
sending each isometry in $E$ to its rotational part.
Then $\rho$ maps the given affine Weyl group $W_{\rm{aff}}$ of $E$ to the Weyl group $W$ of $\partial_T E$.
An oriented geodesic segment $\overline{xy}$ in $E$ corresponds to a point $\xi \in \partial_T E$ looking at infinity. We define the \textit{direction of} $\overline{xy}$ as the image of $\xi$ in $(\partial_T E)/W$. 

This interaction yields the following.

\begin{Prop} \cite [Proposition 4.2.1]{kleiner1997rigidity}
    For a Euclidean building $\cal{B}$, its Tits boundary $\partial_T\cal{B}$ is a spherical building.
\end{Prop}

\subsection{Groups and buildings}

\begin{definition}
    A building is said to be \textit{thick} if every codimension one simplex called \textit{panel} of chambers is a face of at least three chambers.
\end{definition}
A notable feature of thick buildings is that their automorphism groups are rich enough. As in \cite [Section 3.7, Section 4.9]{kleiner1997rigidity}, irreducible buildings of sufficiently high rank can be reduced to thick buildings and they were classified by Jacques Tits. It can be reformulated as follows.
\begin{theorem} [Classification of buildings] \label{classify} 
We have the following one-to-one correspondences up to isomorphisms:
\begin{itemize}
    \item Every thick irreducible spherical building of rank at least $3$ corresponds to a classical group, a simple algebraic group, or a mixed group defined over a field.
    
    \item Every thick irreducible Euclidean building of rank at least $4$ 
    corresponds to a classical group, an absolutely simple algebraic group, or a mixed group defined over a field that is complete with respect to a discrete valuation.

\end{itemize}
\end{theorem}

There are some exceptional cases (in particular, associated with skew fields) but Theorem \ref{classify} captures the essence of the classification and we will mainly focus on the buildings of type $A_n$ associated with fields. 

Theorem \ref{classify} guarantees that thick irreducible higher-rank buildings arise from groups. One reason that buildings are regarded as an analogue of symmetric spaces is that the associated groups act on them nicely.

\begin{proposition}
    For a building $X$, the action of the associated group $G$ on $X$ is simplicial (i.e., sending $i$-simplices to $i$-simplices), isometric, and continuous. 
\end{proposition}

To give a quick and explicit example of this group action, we consider type $A_n$.

\medskip
\textbf{Group actions on spherical buildings of type $A_{n-1}$}
Let $F$ be a field, $G=\rm{SL}_n(F)$, and $\Delta$ be the spherical building of $G$. Then the simplices of $\Delta$ are the flags, which are increasing chains $\{0\}=W_0 \subset W_1 \subset \dots \subset W_k=F^n$ of subspaces of the vector space $F^n$ for $k \leq n$ such that dim$W_i <$ dim $W_{i+1}$ for each $i$. 
Given two flags $\cal{F}_1$ and $\cal{F}_2$, 
$\cal{F}_1$ is contained in $\cal{F}_2$ as simplices in $\Delta$ if the subspaces appearing in $\cal{F}_1$ also appear in $\cal{F}_2$. 
The associated group $G$ acts on the set of flags by left multiplication, i.e., 
$g \cdot (W_0 \subset W_1 \subset \dots \subset W_k)=g \cdot W_0 \subset g \cdot W_1 \subset \dots \subset g \cdot W_k $.
Each flag corresponds to a parabolic subgroup via the stabilizer map. Hence one can identify the simplices of $\Delta$ with the parabolic subgroups as
$$
\cal{F}_1 \subset \cal{F}_2 \Longleftrightarrow
P_1 \supset P_2,
$$
where $P_i$ is the corresponding parabolic subgroup of a flag $\cal{F}_i$ for each $i$.
Thus the top-dimensional simplices (i.e., chambers) are the minimal parabolic subgroups and the $0$-simplices (i.e., vertices) are the maximal parabolic subgroups.
We remark that $G$ acts on the set of parabolic subgroups by conjugation, and if $P$ is the stabilizer of $\cal{F}$, then $gPg^{-1}$ is the stabilizer of $g \cdot\cal{F}$. Hence the action of $G$ on $\Delta$ is well-defined.

\medskip
\textbf{Group actions on Euclidean buildings of type $A_{n-1}$}
Let $F$ be a non-archimedean local field with a valuation $\nu$, $G=\rm{PGL}_n(F)$, and $\cal{B}$ be the Euclidean building of $G$.
Let $\cal{O}=\{x \in  \cal{F} | \nu(x) \geq 0 \}$ be the valuation ring of $F$. 
A \textit{lattice} $\cal{L}$ in the vector space $F^n$ is an $\cal{O}$-module $\cal{O}v_1+ \dots +\cal{O}v_n=\{x_1v_1+\dots+x_nv_n|x_i \in \cal{O}\}$ for a basis $\{v_1, \dots, v_n\}$ for $F^n$. We say that two lattices $\cal{L}_1$ and $\cal{L}_2$ are \textit{equivalent} if $\cal{L}_1=k\cal{L}_2$ for some nonzero $k\in F$. Then the set of equivalence classes of lattices is the set $\cal{B}(0)$ of vertices of $\cal{B}$. By abusing notation, we denote both a lattice and its equivalence class by $\cal{L}$. Two vertices $\cal{L}_1$ and $\cal{L}_2$ are adjacent if 
$\cal{L}_2 \subset \cal{L}_1$ or equivalently  $\pi \cal{L}_1 \subset \cal{L}_2$, where $\pi$ is a fixed uniformizer in $\cal{O}$ i.e., $\nu(\pi)=1$. Inductively, the vertices $\cal{L}_1, \dots, \cal{L}_k$ form a simplex if $\pi \cal{L}_k \subset \cal{L}_1 \subset \cal{L}_2 \dots \subset \cal{L}_k$.
The associated group $G$ acts on $\cal{B}(0)$ by 
$g \cdot (\cal{O}v_1+ \dots +\cal{O}v_n)=\cal{O}(g \cdot v_1)+ \dots \cal{O} (g \cdot v_n)$. In particular, this action is transitive and the \textit{standard lattice} $\cal{O}e_1+ \dots +\cal{O}e_n$ has the stabilizer $K=\rm{PGL}_n(\cal{O})$, which is a compact subgroup of $G$, where $\{e_1, \dots e_n\}$ is the standard basis for $F^n$.
Hence $\cal{B}(0)$ can be identified with $G/K$.    

\subsection{Root groups} 

\begin{definition} \label{root}
    Let \(\Delta\) be a spherical building.
    A \emph{root} in \(\Delta\) is a half-apartment in \(\Delta\) bounded by the wall of a reflection of the Coxeter group.
    Given a root \(\alpha\) in \(\Delta\), 
    we define the \emph{root group} \(U_{\alpha}\) of \(\alpha\) by
      \[
    U_{\alpha} = \left\{ g \in \mathrm{Aut}(\Delta) \ \middle| \
    \begin{array}{l}
    g \text{ fixes every chamber having a panel} \\
    \text{in the interior of } \alpha \text{ pointwise}.
    \end{array}
    \right\}.
    \]
\end{definition}

It follows that the root group $U_{\al}$ fixes $\al$ pointwise.
Given a root of a Lie group $G$, i.e., a linear functional on a maximal abelian subalgebra $\mathfrak{a}$ or the corresponding vector in a maximal torus $A=\rm{exp}(\mathfrak{a})$ (e.g.\@ the diagonal matrices in $G$), the half-apartment in Definition \ref{root} can be thought of as the half-region of $A$ in the direction of the given root.

\begin{example}\label{root gp}
    Let $G=\rm{SL}_3(\mathbb{R})$ which is a Lie group of type $A_2$.
    Then $G$ has six roots in $A \cong \mathbb{R}^2$ with the root groups:
    $$
    \begin{pmatrix} 1 & * & 0\\ 0 & 1 &0 \\ 0& 0& 1\end{pmatrix}, 
    \begin{pmatrix} 1 & 0 & 0\\ 0 & 1 &* \\ 0& 0& 1\end{pmatrix},
    \begin{pmatrix} 1 & 0 & *\\ 0 & 1 &0 \\ 0& 0& 1\end{pmatrix},
    \begin{pmatrix} 1 & 0 & 0\\ * & 1 &0 \\ 0& 0& 1\end{pmatrix},
    \begin{pmatrix} 1 & 0 & 0\\ 0 & 1 &0 \\ 0& *& 1\end{pmatrix},
    \begin{pmatrix} 1 & 0 & 0\\ 0 & 1 &0 \\ *& 0& 1\end{pmatrix}.
    $$

\end{example}

One can observe that each root group is isomorphic to the additive group of $\mathbb{R}$. In general, we have this property for split algebraic groups.
\begin{proposition} 
    The root groups of any $F$-split algebraic group $G$ are isomorphic to the additive group of $F$.
\end{proposition}

 An \textit{algebraic group defined over a field $F$} is a zero locus of finitely many polynomials of coefficients in $F$, and it is said to be \textit{split} if its maximal torus is diagonalizable over $F$. Typical examples of $F$-split algebraic groups include $\mathrm{SL}_n$, $\mathrm{PGL}_n$, $\mathrm{SO}_n$, and $\mathrm{Sp}_{2n}$. We refer the reader to \cite{onishchik2012lie} for a background of algebraic groups.

\subsection{Lattices}
Let $G$ be a locally compact second countable group (e.g. a Lie group). A subgroup $\Ga$ of $G$ is a \textit{lattice} of $G$ if $\Ga$ is discrete and $G/\Ga$ has finite volume with respect to a Haar measure. A lattice $\Ga$ is said to be \textit{cocompact} (or \textit{uniform}) if $G/\Ga$ is compact and \textit{nonuniform} otherwise.

The celebrated Margulis Arithmeticity Theorem \cite{margulis1991discrete} says that irreducible lattices of higher rank semisimple Lie groups are arithmetic subgroups. We remark that the converse also holds by Borel-Harish-Chandra \cite{borel1962arithmetic}. A well-known technique to obtain arithmetic lattices is the \textit{restriction of scalars}: Let $G$ be an algebraic group over a number field $F$ (i.e., $[F:\mathbb{Q}] < \infty)$, and $\cal{O}$ be the ring of integers of $F$. Then $\Ga=G(\cal{O})$ embeds diagonally into $\tilde{G}=\prod_{\sigma}G^\sigma$ and it is an arithmetic subgroup of $\tilde{G}$, where $\sigma$ are the embeddings $F \to \mathbb{C}$ (e.g.
Galois automorphisms of $F$ over $\mathbb{Q}$) and $G^{\sigma}$ is the image of $G$ under $\sigma$. In particular, we can obtain cocompact arithmetic lattices as follows.

\begin{proposition} \cite[Corollary 5.5.10, Appendix C]{morris2015introduction}
    If $G^\sigma$ is compact for some $\sigma$, then the diagonally embedded $\Ga$ is a cocompact arithmetic lattice in $\prod_{\sigma}G^\sigma$.
\end{proposition}

We refer to \cite{morris2015introduction} for a comprehensive background.

\section{Euclidean buildings and associated fields}

For a Euclidean building, its associated group and field coincide with those associated with its spherical building at infinity.  
In particular, the fields are determined by the root groups, which play a crucial role in classifying the buildings.

\begin{prop} \label{field} 
    Let $\Delta$ be a spherical building associated to an $F$-split group $G$ defined over a field $F$.
    If $\Delta'$ is a spherical subbuilding of $\Delta$,
    then the field $F'$ associated with $\Delta'$ is isomorphic to a subfield of $F$.
\end{prop}

\begin{proof}
    Let $\al$ be a root in $\Delta'$. Since $\Delta'$ is a subbuilding of $\Delta$, its apartments are apartments of $\Delta$. Thus $\al$ is a root also in $\Delta$.
    Denote by $U_{\al}'$ and $U_{\al}$ the root groups 
    of $\al$ in $\Delta'$ and $\Delta$, respectively.
    By \cite[Proposition 7.37]{abramenko2008buildings},
    $U_{\al}'$ is a subgroup of $U_{\al}$. In the split case, the root groups are isomorphic to the additive groups of the associated fields. Hence $F'$ is an additive subgroup of $F$. It remains to show that $F'$ is closed under the multiplication of $F$.

    Fix a \textit{root group sequence}\footnote{See \cite[Definition 12.1]{weiss2004structure}. For example, the first three matrices of Example \ref{root gp} form a root group sequence of type $A_2$.} 
    $\Theta'=\{U_i'\}$ from the root group labeling 
    of $\Delta'$. Then $\Delta$ has the root group sequence $\Theta=\{U_i\}$ induced by the same directed edge of the Coxeter diagram. 
    In particular, $U_i'$ and $U_i$ correspond to the same root and $U_i'$ is a non-trivial subgroup of $U_i$ for each $i$.
    Let $\alpha$ and $\beta$ be the roots corresponding to the first and last root groups of $\Theta'$, respectively.
    As in \cite [Proposition 11.22]{weiss2004structure},
    for the root $\be$ in an apartment $\Sigma$,
    define a subgroup $H_{\be}=\langle \mu_(g)\mu(h) | g,h \in {U_{\be}}^*=U_\be \backslash \{e\} \rangle$ of $G$, where $\mu(g)$ is the unique element in ${U_{-\be}}^*g{U_{-\be}}^*$ mapping $\Sigma$ to itself and $\be$ to its opposite root $-\be$ for each $g \in U_{\be}$.
    By \cite[Lemma 12.22]{weiss2004structure}, the action of $H_\be$ on $U_\al$ is generated by the multiplicative maps $x \mapsto xt$ of $F$ for all $t\in F^*$ by identifying $U_\al$ with $F$. 
    Then $H_\beta'=\langle \mu(g')\mu(h') | g',h' \in U_\beta'^* \rangle$ is a subgroup of $H_\beta$ and the action of $H_\beta'$ on $U_\alpha'$ is a restriction of the action of $H_\beta$ on $U_\alpha$. Since the action was the multiplication of $F$, $F'$ is closed under the multiplication of $F$ by identifying $U_\alpha'$ with $F'$.
\end{proof}

Proposition \ref{field} is mentioned in \cite[p.390]{abramenko2008buildings} for general cases without proof. In the non-split cases, the root groups may not be isomorphic to the additive groups of the associated fields. We can still see the additive group of the field in two ways: first by identifying the root groups with vector spaces over the field, and second, by noting that the centers of the root groups are isomorphic to the field.

In \cite{kleiner1997rigidity}, Kleiner-Leeb considered spherical buildings as Tits boundaries, which we denote by $\partial_T X$.
A common idea to prove quasi-isometry rigidity results since the work of Mostow is to get a nice map on Tits boundaries which eventually induces a desired consequence.

\begin{prop}\label{1} 
    Let $X$ and $Y$ be irreducible symmetric spaces or Euclidean buildings of  equal rank $r \geq 4$ associated with split groups. Assume that the type of $(X,Y)$ is none of the following: 
    \[
    (A_n,B_n), (A_n,C_n), (A_n,F_4), (D_4, A_4), \text{and }(\text{any type},G_2).
    \]
   If there is a quasi-isometric embedding $X \to Y$, then the field associated with $X$ is isomorphic to a subfield of the field associated with $Y$. 
\end{prop}

\begin{proof}
    Let $\phi:X \rightarrow Y$ be a quasi-isometric embedding. 
    As shown in \cite [Proof of Theorem 1.8]{fisher2018quasi}, the boundary $\partial_T \phi(X)$ at infinity of $\phi(X)$  is a top-dimensional spherical subbuilding of $\partial_T Y$. 
    By \cite [Theorem 3.1] {kleiner2006rigidity}, such a subbuilding is of the form $\partial_T \phi(X)$=$\partial_T Y'$ for some symmetric subspace or Euclidean subbuilding $Y'$ of $Y$. 
    Then $\phi(X)$ and $Y'$ are at bounded distance so we have a quasi-isometry $\phi_0: X \to Y'$ by restricting the target of $\phi$.

    It follows from \cite[Proof of Theorem 1.1.3]{kleiner1997rigidity} that 
    the quasi-isometry $\phi_0: X \to Y'$ induces an isomorphism 
    $\partial_T X \to \partial_T Y'$ of the spherical buildings at infinity. Indeed, it is shown that such $\phi_0$ is a homothety with a fixed additive error, and so we may assume that $\phi_0$ is a $(1,c)$-quasi-isometry for some fixed $c>0$ by rescaling the metrics. Then such a $(1,c)$-quasi-isometry induces a map on the spherical buildings at infinity which preserves the Tits metrics. This induced map is the desired isomorphism on the Tits boundaries.
    
    By Theorem \ref{classify}, the field $F_X$ associated with $\partial_T X$ is isomorphic to the field $F_{Y'}$ associated with $\partial_T Y'$. In particular, those fields are also associated with $X$ and $Y'$, respectively. Since $\partial_T Y'$ is a spherical subbuilding of $\partial_T Y$, Proposition \ref{field} implies that $F_{Y'}$ is a subfield of $F_Y$. Hence $F_X$ is isomorphic to the subfield $F_{Y'}$ of $F_Y$.
\end{proof}

The assumption rank $r \geq 4$ is necessary to use the classification of buildings Theorem \ref{classify}. We exclude some types of $(X,Y)$ to use \cite[Proof of Theorem 1.8]{fisher2018quasi}, which requires all Weyl pattern embeddings to be conformal. The list of types that fail to satisfy this condition is provided in \cite[Corollary 1.10]{fisher2018quasi}. 
We compare Ramanujan complexes that are constructed from Euclidean buildings of type $A_n$. Hence we will apply Proposition \ref{1} to the type $(A_n,A_n)$.

\section{From box spaces to Ramanujan complexes} \label{ra}

Let $G$ and $H$ be finitely generated residually finite groups with decreasing sequences $\{N_i\}_i$ and $\{K_i\}_i$ of finite index normal subgroups of $G$ and $H$, respectively, with trivial intersection. 
Heuristically, $G/N_i$ is similar to $G$ around small neighborhoods of the identities and such neighborhoods get larger as $i$ increases.
Hence $G/N_i$ looks more like $G$ as $i \to \infty$. Based on this, Khukhro-Valette showed the following box space rigidity.

\begin{prop} [Box space rigidity] \cite [Proposition 16] {khukhro2017expanders} \label{KV}
    Let $G$, $H$, $\{N_i\}_i$,  and $\{K_i\}_i$ be defined as above. 
    If the box space $\{\mathrm{Cay}(G/N_i)\}_i$ is quasi-isometric to 
    $\{\mathrm{Cay}(H/K_i)\}_i$, then $G$ is quasi-isometric to $H$.
\end{prop}

\begin{Rmk} \label{KV2}
\begin{enumerate}
    \item Khukhro-Valette show Proposition \ref{KV} for quasi-isometries but the same proof works for quasi-isometric embeddings as well. The embedding maps with the common quasi-isometry constants are obtained by the diagonal limit argument in their proof. 

    \item Their definition of quasi-isometry between two families of metric spaces is with respect to the \textit{metrized disjoint union}, that is considering the distances between different components as well. In \cite [Lemma 3]{khukhro2017expanders}, they derive the component-to-component metric as in Definition \ref{def}.

    \item The original statement is written for a more general setting, namely, for any finitely generated groups $G$ and $H$, and any sequences converging to $G$ and $H$ with respect to Chabauty topology, rather than just finite quotients.

\end{enumerate}
\end{Rmk}

We can approximate Euclidean buildings by their 1-skeleton graphs similarly to \cite[Proposition 8.45]{MR1744486} considering the set of all chambers as a partition of a building.
\begin{lemma}\label{skeleton}
    Let $\cal{B}$ be a Euclidean building and $\cal{B}(0)$ be the set of all vertices equipped with the graph metric i.e., the length of shortest paths in the $1$-skeleton graph of $\cal{B}$.
    Then $\cal{B}(0)$ is $(3,c)$-quasi isometric to $\cal{B}$, where $c$ is the diameter of a chamber in $\cal{B}$.
\end{lemma}
\begin{proof}
    We will show that the inclusion map $i: \cal{B}(0) \hookrightarrow \cal{B}$ is a $(3,c)$-quasi-isometry.
    Denote by $d$ and $d_{g}$ the metric on $\cal{B}$ and the graph metric on $\cal{B}(0)$, respectively.
    Since any point in $B$ is contained in a chamber, $i$ is $c$-coarsely surjective.
    The right inequality of the following is also obvious as $d(u,v) \leq d_g(u,v)$ for all $u,v \in \cal{B}(0)$.
    $$
    \frac{1}{3}d_g(u,v)-c \leq d(i(u),i(v)) \leq 3d_g(u,v)+c 
    $$
    What remains is to show the left inequality. Given $u$, $v \in \cal{B}(0)$, set $N$ to be the smallest positive integer such that $d(u,v)+c < N$.
    Consider a path $\ga$ from $u$ to $v$ in $\cal{B}$ whose length is $d(u,v)$.  Pick $N$ distinct points $p_1, \dots ,p_N $ on $\ga$ with $p_1=u$ and $p_{N}=v$. Since the set of all chambers forms a partition of $\cal{B}$, for each $p_i$, there exists $q_i \in \cal{B}(0)$ such that $p_i$ and $q_i$ are in the same chamber so that $d(p_i,q_i)<c$.
    Then we have 
    $$
    d(q_i,q_{i+1}) \leq d(q_i,p_i)+d(p_i,p_{i+1})+d(p_{i+1},q_{i+1})
    \leq 2c + \frac{d(u,v)}{N} \leq 3c 
    $$
    for every $i$. We remark that $c \geq 1$ by setting the length of the edges (i.e., 1-simplices) to be 1 with respect to $d$.
    Since $c$ is the diameter of the chambers,
    for the vertices $q_i$ and $q_{i+1}$, there exist consecutively adjacent three chambers $C_1$, $C_2$, and $C_3$ (not necessarily distinct) such that $q_i \in C_1$ and $q_{i+1} \in C_3$.
    This implies that $d_g(q_i,q_{i+1}) \leq 3 $ 
    and hence we obtain $d_g(u,v)\leq 3(N-1)$ by taking $q_1=u$ and $q_{N}=v$.
    By the minimality of $N$, $d_g(u,v) \leq 3(d(u,v)+c)$, as desired.  
\end{proof}

We have the following Euclidean building version of Proposition \ref{KV}.

\begin{prop}\label{2}
    Let $\cal{B}$ and $\cal{C}$ be Euclidean buildings of locally compact second countable groups $G$ and $H$, respectively. 
    Let $\Ga$ and $\La$ be cocompact lattices of $G$ and $H$, and $\{\Ga_i\}$ and 
    $\{\La_i\}$ be decreasing sequences of finite index normal subgroups of $\Ga$ and $\La$, respectively. 
    If an infinite family $\{ \Ga_i\backslash \cal{B} \}_i$ of quotient buildings quasi-isometrically embeds in $\{ \La_i\backslash \cal{C} \}_i$,
    then there exists a quasi-isometric embedding $\cal{B} \to \cal{C}$.
\end{prop} 

\begin{proof} 
    Denote by $\cal{B}(0)$ the set of vertices of $\cal{B}$.
    We equip $\cal{B}(0)$ with the graph metric.
    Identifying $\cal{B}(0)$ with $G/K$ for a compact subgroup $K$ of $G$, a cocompact lattice $\Ga$ of $G$ acts on $\cal{B}(0)$ isometrically, properly discontinuously, and cocompactly. By Milnor-Švarc, $\Ga$ is quasi-isometric to $\cal{B}(0)$. By Lemma \ref{skeleton},  $\Ga$ and $\cal{B}$ are quasi-isometric.
    So are their quotients $\Ga /\Ga_i$ and $ \Ga_i \backslash \cal{B}$ by modding out the finite index normal subgroups $\Ga_i$ of $\Ga$ for each $i$. 
    Note that the quasi-isometry constants do not depend on $i$ but only on the diameter of the chambers of $\cal{B}$.
    Then given quasi-isometric embeddings
    $ \Ga_i \backslash \cal{B} \to
   \La_i \backslash \cal{C}$ of quotient
    buildings induce quasi-isometric embeddings
    $\Ga/ \Ga_i \to \La /\La_i$ of quotient groups for each $i$.
    By Proposition \ref{KV} and Remark \ref{KV2}, this yields a quasi-isometric embedding $\Ga \to \La$. Again by Milnor-Švarc, we complete the proof.
\end{proof}

 Before we review how $\Ga$ and $\Ga(I_i)$
 were constructed for Ramanujan complexes by Lubotzky-Samuels-Vishne \cite{lubotzky2005explicit}, we remark that many details are omitted. 
 In particular, the construction originally begins with a $k$-algebraic group not with $\rm{PGL}$.
 See also \cite{lubotzky2014ramanujan}.
 The only necessary characteristics of $\Ga$ and $\Ga(I_i)$ to prove Theorem \ref{3} is that $\Ga$ is a cocompact lattice and $\Ga(I_i)$ is a finite index normal subgroup of $\Ga$ for each $i$.

\bigskip
\textbf{Construction of Ramanujan complexes by \cite{lubotzky2005explicit}}
    Let $d \geq 3$, $\mathbb{G}=\rm{PGL}_d$, $p$ be a prime, $k=\mathbb{F}_p(y)$ be the global function field over $\mathbb{F}_p$, and $S=\{\nu_y\}\cup T =\{\nu_y,\nu_{y+1},\nu_{\frac{1}{y}}\}$ be the finite set of valuations of $k$, where $\nu_y(a_my^m+\dots+a_ny^n)=m$ for $a_m \neq 0$ and $m \leq n$.
    Then $\mathcal{O}_S=\mathbb{F}_p[\frac{1}{y},\frac{1}{y+1},y]$ is the ring of $S$-integers in $k$. As in the Restriction of Scalars, $\mathbb{G}(\mathcal{O}_S)$ embeds diagonally into $\prod_{\nu \in S}\mathbb{G}(k_\nu)$ as a lattice, where $k_\nu$ denotes the completion of $k$ with respect to the valuation $\nu$. For instance, $k_{\nu_{y}}$ is the local field $F=\mathbb{F}_p((y))$ of Laurent series over $\mathbb{F}_p$ and so $\mathbb{G}(k_{\nu_{y}})=\rm{PGL}_d(F)$. In particular, $T$ was taken so that $\mathbb{G}(k_\nu)$ is compact for $\nu \in T$. Hence the projection $\tilde{\Gamma}$ of $\mathbb{G}(\mathcal{O}_S)$ in $\rm{PGL}_d(F)$ is an $S$-arithmetic cocompact lattice in $\rm{PGL}_d(F)$. 

    On the other hand, let $\mathcal{B}$ be the Euclidean building of $G=\rm{PGL}_d(F)$, $\mathcal{O}=\mathbb{F}_p[[y]]$ be the ring of integers in $F$, $K=\rm{PGL}_d(\mathcal{O})$, and  $\mathcal{B}(0)$ be the set of vertices of $\mathcal{B}$. One can identify $\mathcal{B}(0)$ with $G/K$ and assign $d$ colors $\tau:\mathcal{B}(0) \to \mathbb{Z}/d\mathbb{Z}$ by $\tau(gK)=\nu_y(\text{det } g)$ (mod $d$).
    Let $\Gamma$ be the subgroup of $\tilde{\Gamma}$ generated by the elements $\gamma \in \tilde{\Gamma}$ such that $\tau(\gamma K)-\tau(e K)=1$ (mod $d$).
    Then $\Ga$ is a finite index subgroup of $\tilde{\Ga}$ and so a cocompact lattice of $G$.
    They showed that the finite quotient $\Gamma(I) \backslash \mathbb{B}$ is a Ramanujan complex, where 
    $\Gamma(I)=\text{ker}(\Gamma \to G(\mathcal{O}_S/I))$ is the finite index congruence subgroup of $\Gamma$ for any ideal $0 \neq I \triangleleft \mathcal{O}_S$. One can obtain an infinite family $\mathscr{X}_p= \{ \Gamma(I_i)\backslash \mathcal{B}\}_i$ of Ramanujan complexes by taking infinitely many ideals $I_i\triangleleft \mathcal{O}_S$.


\begin{proof} [\textbf{Proof of Theorem \ref{3}}]
    Given $d \geq 4$ and a prime $p$, denote by $F_p$ and $\mathcal{B}_p$ the field $\mathbb{F}_p((y))$ and the Euclidean building of $\rm{PGL}_d(\mathbb{F}_p((y)))$, respectively. 
    For distinct primes $p$ and $q$, the field $F_p$ and any subfield of $F_q$ are not isomorphic since they have different characteristics. 
    By Proposition \ref{1}, there is no quasi-isometric embedding $\mathcal{B}_p \to \mathcal{B}_q$. 
    As $\mathscr{X}_p= \{\Gamma(I_i)\backslash \mathcal{B}_p\}_i$, Proposition \ref{2} implies that $\mathscr{X}_p$ does not quasi-isometrically embed in $\mathscr{X}_q$. 
\end{proof}

\begin{proof} [\textbf{Proof of Corollary \ref{coro}}]
    The geometric expanders part is obvious. For the topological expanders part, we remark that the skeleta of a building are quasi-isometric to the building by an inductive argument of Lemma \ref{skeleton}. 
    Hence if two quotient buildings of $(d-1)$-dimension do not quasi-isometrically embed into one another, then neither do their $(d-2)$-skeleta.
\end{proof}

\section{Further remarks}\label{5}

We distinguished Ramanujan complexes associated with different primes. A natural question would be the case from a single prime.

\begin{question}
Fix a prime $p$.
Consider two Ramanujan complexes 
$\{\Gamma_i \backslash \mathcal{B}_p\}_i$ and 
$\{\Lambda_i \backslash \mathcal{B}_p\}_i$, 
where $\Gamma_i$ and $\Lambda_i$ are finite index normal subgroups of the cocompact lattice $\Gamma$ of $\rm{PGL}_d(\mathbb{F}_p ((y)))$.
Under what conditions there is no
quasi-isometric embedding $\{\Gamma_i \backslash \mathcal{B}_p\}_i \to \{\Lambda_i \backslash \mathcal{B}_p\}_i$?
In particular, if $[\Gamma : \Gamma_i]$ grows much faster than $[\Gamma : \Lambda_i]$, can we say 
they do not admit quasi-isometric embeddings?
\end{question}

One may consider the case in which systoles of 
$\{\Gamma_i \backslash \mathcal{B}_p\}_i$ grow much faster than 
systoles of $\{\Lambda_i \backslash \mathcal{B}_p\}_i$. 
A possible approach consists of the following steps:
\begin{enumerate}
    \item Rephrase the finite indices $[\Gamma: \Gamma_i]$ of cocompact lattices in terms of the systoles of $\Gamma_i \backslash \mathcal{B}_p$.

    \item Show that the difference in growth rates of systoles implies that there is no quasi-isometric embedding $\{\Gamma_i \backslash \mathcal{B}_p\}_i \to \{\Lambda_i \backslash \mathcal{B}_p\}_i$.
\end{enumerate}

The first step seems plausible if one could generalize 
the following result of Belolipetsky-Weinberger for symmetric spaces and arithmetic cocompact lattices to the setting of Euclidean buildings and $S$-arithmetic lattices.

\begin{theorem} \textup{\cite{belolipetsky2024growth}} \label{bw}
Let $X$ be the symmetric space of a semisimple Lie group $G$ of rank at least $2$. Suppose that $\Gamma \backslash X$ is a compact arithmetic locally symmetric space and 
$\{\Gamma_i \backslash X \to \Gamma \backslash X\}_i$ is a sequence of congruence coverings of degree $d_i=[\Gamma: \Gamma_i]$ for $i=1,2, \dots$. Then the absolute $k$-dimensional systole of $\Gamma_i \backslash X$ grows polylogarithmically with $d_i$ for 
$1 \leq k \leq \text{max} \{1, r_1\}$, where $r_1$ is the strongly orthogonal rank of $\Gamma \backslash X$.
     
\end{theorem}

For the second step, \textit{separation profiles} introduced in \cite{benjamini2012separation} might be used. Separation profiles measure how difficult it is to disconnect the objects, and hence systoles may be related to them.
It has been shown that separation profiles are invariant under coarse embeddings and under quasi-isometric embeddings as well.

\bigskip

On the other hand, an open problem concerning expanders is whether all expanders are super-expanders. 
Since being expanders (resp. super-expanders) implies that they coarsely embed into any Hilbert space (resp. uniformly convex Banach spaces), 
the problem can be reformulated in terms of coarse embeddability into uniformly convex Banach spaces.
Hence distinguishing expanders up to coarse embeddings reduces the problem and this motivated Question\ref{question}.

By a recent result of Bensaid \cite{bensaid2022coarse}, we could observe that lower rank Ramanujan complexes do not coarsely embed into higher rank Ramanujan complexes as follows.
\begin{proposition}
    For $d \geq3$ and a prime $p$, denote by $\mathcal{B}_{d,p}$  the Euclidean building of $\mathrm{PGL}_d(\mathbb{F}_p((y)))$ and
    $\mathscr{X}_{d,p}=\{\Gamma_i \backslash \mathcal{B}_{d,p}\}_i$ a family of Ramanujan complexes.
  If $d > d'$, then there is no coarse embedding $\mathscr{X}_{d,p} \to \mathscr{X}_{d',p'}$ for any primes $p$ and $p'$.
\end{proposition}

\begin{proof}
   By \cite[Proposition 3.1]{das2018geometry}, we can replace quasi-isometric embeddings by coarse embeddings in Proposition\ref{KV2}. 
   Since $\Gamma$ and $\mathcal{B}_{d,p}$ are quasi-isometric as shown in the proof of Proposition \ref{2}, Proposition \ref{2} also holds for coarse embeddings.
   If $d>d'$, then \cite[Theorem 1.4]{bensaid2022coarse} implies that there is no coarse embedding $\mathcal{B}_{d,p} \to \mathcal{B}_{d',p'}$  and this completes the proof. 
\end{proof}

\begin{question}
    What happens in the equal-rank cases? Given two families of Ramanujan complexes of equal rank, can one distinguish them up to coarse embedding?
\end{question}

\bibliographystyle{alpha}
\bibliography{refs}

\newcommand{\etalchar}[1]{$^{#1}$}
\begin{thebibliography}{KdLdlS19}

\bibitem[AB08]{abramenko2008buildings}
Peter Abramenko and Kenneth~S Brown.
\newblock {\em Buildings: theory and applications}, volume 248.
\newblock Springer Science \& Business Media, 2008.

\bibitem[Alo86]{alon1986eigenvalues}
Noga Alon.
\newblock Eigenvalues and expanders.
\newblock {\em Combinatorica}, 6(2):83--96, 1986.

\bibitem[Ben22]{bensaid2022coarse}
Oussama Bensaid.
\newblock Coarse embeddings of symmetric spaces and {E}uclidean buildings.
\newblock {\em arXiv preprint arXiv:2201.06442}, 2022.

\bibitem[BH99]{MR1744486}
Martin~R. Bridson and Andr\'e Haefliger.
\newblock {\em Metric spaces of non-positive curvature}, volume 319 of {\em Grundlehren der mathematischen Wissenschaften [Fundamental Principles of Mathematical Sciences]}.
\newblock Springer-Verlag, Berlin, 1999.

\bibitem[BHC62]{borel1962arithmetic}
Armand Borel and Harish-Chandra.
\newblock Arithmetic subgroups of algebraic groups.
\newblock {\em Annals of mathematics}, 75(3):485--535, 1962.

\bibitem[BST12]{benjamini2012separation}
Itai Benjamini, Oded Schramm, and {\'A}d{\'a}m Tim{\'a}r.
\newblock On the separation profile of infinite graphs.
\newblock {\em Groups, Geometry, and Dynamics}, 6(4):639--658, 2012.

\bibitem[BW24]{belolipetsky2024growth}
Mikhail Belolipetsky and Shmuel Weinberger.
\newblock Growth of k-dimensional systoles in congruence coverings.
\newblock {\em Geometric and Functional Analysis}, 34(4):979--1005, 2024.

\bibitem[Das18]{das2018geometry}
Kajal Das.
\newblock From the geometry of box spaces to the geometry and measured couplings of groups.
\newblock {\em Journal of Topology and Analysis}, 10(02):401--420, 2018.

\bibitem[dLV19]{de2019superexpanders}
Tim de~Laat and Federico Vigolo.
\newblock Superexpanders from group actions on compact manifolds.
\newblock {\em Geometriae Dedicata}, 200:287--302, 2019.

\bibitem[EF97]{eskin1997quasi}
Alex Eskin and Benson Farb.
\newblock Quasi-flats and rigidity in higher rank symmetric spaces.
\newblock {\em Journal of the American Mathematical Society}, pages 653--692, 1997.

\bibitem[EK24]{evra2024bounded}
Shai Evra and Tali Kaufman.
\newblock Bounded degree cosystolic expanders of every dimension.
\newblock {\em Journal of the American Mathematical Society}, 37(1):39--68, 2024.

\bibitem[Evr17]{evra2017finite}
Shai Evra.
\newblock Finite quotients of {B}ruhat--{T}its buildings as geometric expanders.
\newblock {\em Journal of Topology and Analysis}, 9(01):51--66, 2017.

\bibitem[FGL{\etalchar{+}}12]{fox2012overlap}
Jacob Fox, Mikhail Gromov, Vincent Lafforgue, Assaf Naor, and J{\'a}nos Pach.
\newblock Overlap properties of geometric expanders.
\newblock {\em Journal f{\"u}r die reine und angewandte Mathematik (Crelles Journal)}, 2012(671):49--83, 2012.

\bibitem[FNvL19]{fisher2019rigidity}
David Fisher, Thang Nguyen, and Wouter van Limbeek.
\newblock Rigidity of warped cones and coarse geometry of expanders.
\newblock {\em Advances in Mathematics}, 346:665--718, 2019.

\bibitem[FS96]{farb1996large}
Benson Farb and Richard Schwartz.
\newblock The large-scale geometry of {H}ilbert modular groups.
\newblock {\em Journal of Differential Geometry}, 44(3):435--478, 1996.

\bibitem[FW18]{fisher2018quasi}
David Fisher and Kevin Whyte.
\newblock Quasi-isometric embeddings of symmetric spaces.
\newblock {\em Geometry \& Topology}, 22(5):3049--3082, 2018.

\bibitem[Gro10]{gromov2010singularities}
Mikhail Gromov.
\newblock Singularities, expanders and topology of maps. {P}art 2: {F}rom combinatorics to topology via algebraic isoperimetry.
\newblock {\em Geometric and Functional Analysis}, 20(2):416--526, 2010.

\bibitem[HLW06]{hoory2006expander}
Shlomo Hoory, Nathan Linial, and Avi Wigderson.
\newblock Expander graphs and their applications.
\newblock {\em Bulletin of the American Mathematical Society}, 43(4):439--561, 2006.

\bibitem[KdLdlS19]{khukhro2019mini}
Ana Khukhro, Tim de~Laat, and Mikael de~la Salle.
\newblock {M}ini-{W}orkshop: {S}uperexpanders and {T}heir {C}oarse {G}eometry.
\newblock {\em Oberwolfach Reports}, 15(2):1117--1160, 2019.

\bibitem[KKL16]{kaufman2016isoperimetric}
Tali Kaufman, David Kazhdan, and Alexander Lubotzky.
\newblock Isoperimetric inequalities for {R}amanujan complexes and topological expanders.
\newblock {\em Geometric and Functional Analysis}, 26(1):250--287, 2016.

\bibitem[KL97]{kleiner1997rigidity}
Bruce Kleiner and Bernhard Leeb.
\newblock Rigidity of quasi-isometries for symmetric spaces and {E}uclidean buildings.
\newblock {\em Publications Math{\'e}matiques de l'IH{\'E}S}, 86:115--197, 1997.

\bibitem[KL06]{kleiner2006rigidity}
Bruce Kleiner and Bernhard Leeb.
\newblock Rigidity of invariant convex sets in symmetric spaces.
\newblock {\em Inventiones Mathematicae}, 163(3):657, 2006.

\bibitem[KV17]{khukhro2017expanders}
Ana Khukhro and Alain Valette.
\newblock Expanders and box spaces.
\newblock {\em Advances in Mathematics}, 314:806--834, 2017.

\bibitem[Li04]{li2004ramanujan}
Wen-Ching~Winnie Li.
\newblock Ramanujan hypergraphs.
\newblock {\em Geometric \& Functional Analysis GAFA}, 14:380--399, 2004.

\bibitem[LP20]{lubotzky2020ramanujan}
Alexander Lubotzky and Ori Parzanchevski.
\newblock From {R}amanujan graphs to {R}amanujan complexes.
\newblock {\em Philosophical Transactions of the Royal Society A}, 378(2163):20180445, 2020.

\bibitem[LPS88]{lubotzky1988ramanujan}
Alexander Lubotzky, Ralph Phillips, and Peter Sarnak.
\newblock Ramanujan graphs.
\newblock {\em Combinatorica}, 8(3):261--277, 1988.

\bibitem[LSV05]{lubotzky2005explicit}
Alexander Lubotzky, Beth Samuels, and Uzi Vishne.
\newblock Explicit constructions of {R}amanujan complexes of type $\tilde{A_d}$.
\newblock {\em European Journal of Combinatorics}, 26(6):965--993, 2005.

\bibitem[Lub94]{lubotzky1994discrete}
Alex Lubotzky.
\newblock {\em Discrete groups, expanding graphs and invariant measures}, volume 125.
\newblock Springer Science \& Business Media, 1994.

\bibitem[Lub14]{lubotzky2014ramanujan}
Alexander Lubotzky.
\newblock Ramanujan complexes and high dimensional expanders.
\newblock {\em Japanese Journal of Mathematics}, 9:137--169, 2014.

\bibitem[Lub18]{lubotzky2018high}
Alexander Lubotzky.
\newblock High dimensional expanders.
\newblock In {\em Proceedings of the international congress of mathematicians: Rio de Janeiro 2018}, pages 705--730. World Scientific, 2018.

\bibitem[Mar73]{gregory1973margulis}
Gregory Margulis.
\newblock Margulis: {E}xplicit constructions of expanders.
\newblock {\em Problemy Peredaci Informacii}, 9(4):71--80, 1973.

\bibitem[Mar91]{margulis1991discrete}
Gregory~A Margulis.
\newblock {\em Discrete subgroups of semisimple {L}ie groups}, volume~17.
\newblock Springer Science \& Business Media, 1991.

\bibitem[MN14]{mendel2014nonlinear}
Manor Mendel and Assaf Naor.
\newblock Nonlinear spectral calculus and super-expanders.
\newblock {\em Publications math{\'e}matiques de l'IH{\'E}S}, 119:1--95, 2014.

\bibitem[Mor15]{morris2015introduction}
Dave~Witte Morris.
\newblock {\em Introduction to arithmetic groups}, volume~2.
\newblock Deductive Press Lieu de publication inconnu, 2015.

\bibitem[NS17]{nowak2017warped}
Piotr Nowak and Damian Sawicki.
\newblock Warped cones and spectral gaps.
\newblock {\em Proceedings of the American Mathematical Society}, 145(2):817--823, 2017.

\bibitem[Ost13]{ostrovskii2013metric}
Mikhail~I Ostrovskii.
\newblock {\em Metric embeddings: {B}ilipschitz and coarse embeddings into {B}anach spaces}, volume~49.
\newblock Walter de Gruyter, 2013.

\bibitem[OV12]{onishchik2012lie}
Arkadij~L Onishchik and Ernest~B Vinberg.
\newblock {\em Lie groups and algebraic groups}.
\newblock Springer Science \& Business Media, 2012.

\bibitem[Vig19]{vigolo2019discrete}
Federico Vigolo.
\newblock Discrete fundamental groups of warped cones and expanders.
\newblock {\em Mathematische Annalen}, 373(1):355--396, 2019.

\bibitem[Wei04]{weiss2004structure}
Richard~M Weiss.
\newblock {\em The structure of spherical buildings}.
\newblock Princeton University Press, 2004.

\end{thebibliography}

\end{document}